\documentclass[reqno,oneside,10pt]{amsproc}

\usepackage{graphicx}
\usepackage{enumerate}
\usepackage[all]{xy}
\usepackage{amsmath,amsfonts,amssymb}
\usepackage[latin9]{inputenc}
\newcommand{\N}{{\mathbb N}}
\newcommand{\Z}{{\mathbb Z}}

\newcommand{\K}{\mathbf{k}}
\newcommand{\sd}{(S^*)^{-1}}

\newcommand{\acts}{\rhd}
\newcommand{\prho}{\overline{\rho}} % partial coaction 

\newtheorem{lemma}{Lemma}

\newtheorem{thm}{Theorem}
\newtheorem{defi}{Definition}

\theoremstyle{definition}
\newtheorem{ex}{Example}

\newcommand{\vai}{\rightarrow}

\newcommand{\id}{{\mathbf 1}}

\newcommand{\B}{\Lambda}
\newcommand{\C}{A}
\newcommand{\D}{\widehat{D}}
\newcommand{\gen}{B}

\newcommand{\Hom}{\operatorname{Hom}}
\newcommand{\End}{\operatorname{End}}

\newcommand{\um }{1_A}
\usepackage{latexsym}

\usepackage{xspace}
\usepackage{amscd}
\usepackage{amssymb}
\usepackage{amsfonts}

\newcommand{\e}{E}
\newcommand{\f}{F}

\newcommand{\hits}{\rhd}

\title[Globalization theorems for partial (co)actions...]{Globalization theorems for partial Hopf (co)actions, and some of their applications}

\author[M.M.S. Alves]{Marcelo \ Muniz \ S. \ Alves}
\address{Departamento de Matem\'atica, Universidade Federal do Paran\'a, Brazil}
\email{marcelomsa@ufpr.br}
\author[E. Batista]{Eliezer \ Batista}
\address{Departamento de Matem\'atica, Universidade Federal de Santa Catarina, Brazil}
\email{ebatista@mtm.ufsc.br}

\thanks{\\ {\bf 2000 Mathematics Subject Classification}: Primary 16W30; Secondary 57T05, 16S40, 16S35.\\   {\bf Key words and phrases:} partial Hopf action, partial group action, partial smash product. }

\begin{document}

\begin{abstract}

Partial actions of Hopf algebras can be considered as a generalization of partial actions of groups on algebras. Among important properties of partial Hopf actions, it is possible to prove the existence of enveloping actions, i.e., every partial Hopf action on an algebra $A$ is induced by a Hopf action on an algebra $B$ that contains $A$ as a right ideal. This globalization theorem allows to extend several results from the theory of partial group actions to the Hopf algebraic setting. In this article, we prove a dual version of the globalization theorem: that every partial coaction of a Hopf algebra admits an enveloping coaction. We also show how this works on a series of examples which go beyond partial group actions. Finally, we explore some consequences of globalization theorems in order to present versions of the duality theorems of Cohen-Montgomery and Blattner-Montgomery for partial Hopf actions.
\end{abstract}

\maketitle

\section{Partial Actions and Coactions}
The notion of partial group actions is a quite well stablished algebraic concept by now. It was originated in the theory of operator algebras in order to classify certain C$^*$-algebras generated by partial isometries \cite{ruy}. After \cite{dok}, partial group actions received a purely algebraic formulation allowing several later developments, including Galois theory for partial group actions on rings \cite{paques}. Roughly speaking, a partial action of a group $G$ on a, not necessarily unital, $\K$-algebra $A$ is a family of ideals $\{D_g \}_{g\in G}$ together with a family of algebra isomorphisms $\{ \alpha_g :D_{g^{-1}}\rightarrow D_g |g\in G \}$ satisfying the following properties:
\begin{enumerate}[(i)]
\item $D_e =A$ and $\alpha_e =\mbox{Id}_A$, where $e$ is the identity of $G$.
\item $\alpha_g (D_{g^{-1}}\cap D_h )=D_g \cap D_{gh}$, for every $g,h\in G$.
\item $\alpha_h \circ \alpha_g (x) =\alpha_{hg}(x)$, for exery $x\in D_{g^{-1}}\cap D_{g^{-1} h^{-1}}$.
\end{enumerate}
A particularly important case of partial group action, from the purely algebraic point of view, is the case where $A$ is unital and all the ideals $D_g \trianglelefteq A$ are generated by central idempotents, thus making them unital ideals. In this case, it is possible to show that this partial action can be thought  as a restriction of a global action of the same group on a larger algebra. This is the globalization theorem as one can find in \cite{dok}. Basically, the globalization theorem states that given a partial action of a group $G$ on a unital algebra $A$ such that every ideal $D_g \trianglelefteq A$ is unital, it is possible to construct a new algebra $B$ (which is not necessarily unital), a group action $\beta :G\rightarrow \mbox{Aut}(B)$ and a monomorphism of algebras $\varphi :A\rightarrow B$ such that
\begin{enumerate}[\bf 1)]
\item $\varphi (A)\trianglelefteq B$.
\item $\varphi (D_g )=\varphi (A) \cap \beta_g (\varphi (A))$.
\item $\beta_g (\varphi (x))=\varphi (\alpha_g (x))$, for every $x\in D_{g^{-1}}$.
\item $B=\sum_{g\in G} \beta_g (\varphi (A))$.
\end{enumerate}
$B$ induces a partial $G$-action $\alpha^\prime = \left (\{D^\prime_g\}_{g \in G}, \{\alpha^\prime_g\}_{g \in G} \right)$ on $\varphi(A)$ by taking $D^\prime_g = \varphi (A) \cap \beta_g (\varphi (A))$ and $\alpha^\prime_g = \beta_g$ restricted to $D_{g^{-1}}^\prime$. Conditions (2) and (3) say that the partial actions on $A$ and $\varphi(A)$ are \emph{equivalent}, and conditions (1) and (4) say that the induced partial action $\alpha^\prime$ on $\varphi(A)$ is \emph{admissible}. In the same paper it is proved that,   under these assumptions, this globalization is unique up to isomorphism; it is also shown that every finite sum of ideals $\varphi (A) \cap \beta_g (\varphi (A))$ has a unity and hence $B$ is a unital algebra when $G$ is finite.

The main idea of partial Hopf actions is to generalize some results of partial group actions in the context of Hopf algebra theory. The concept of partial actions and coactions of Hopf algebras on algebras were introduced by Caenepeel and Janssen in 
\cite{caen06}, generalizing left $H$-module algebras and right $H$-comodule algebras respectively. Their principal motivation was to put the Galois theory for partial group actions on rings into a broader context, namely, the partial entwining structures. In searching for a consistent definition of partial entwining structure, the authors proposed what is meant to be a partial action of a Hopf algebra $H$ on an unital algebra $A$ (in what follows, every algebra is an algebra over a field $\K$).

\begin{defi} A (left) {\bf partial action} of the Hopf algebra $H$ on the 
algebra $A$ is a linear mapping $\alpha: H \otimes A \vai A$, denoted 
here by $\alpha(h \otimes a) = h \cdot a$, such that  
\begin{enumerate}[\bf 1)]
\item $h \cdot (ab) = \sum (h_{(1)} \cdot a) (h_{(2)} \cdot b),$
\item $1_H \cdot a = a,$
\item $h \cdot (g \cdot a) = \sum (h_{(1)} \cdot \um) ((h_{(2)} g) \cdot a).$
\end{enumerate}
\end{defi}
In this case, we call $A$ a (left) {\bf partial} $H$-{\bf module algebra}. We consider only left partial actions in this paper. 

\begin{defi} A morphism of partial $H$-module algebras is an algebra map 
$\theta: A \vai A'$ such that $\theta(h \cdot a) = h \cdot \theta (a)$ for all 
$h \in H$ and all $a \in A$. 
\end{defi}

Note that $A$ is not a $H$-module, nor does $H$ measure $A$. Such partial action can be obtained in the following manner: if $B$ is an $H$-module algebra that has a right ideal $A = \um B$, generated by an idempotent $\um$, such that $A$ is a unital algebra (with unity $\um$), then $A$ becomes a partial $H$-module algebra via the map 
\[
h \cdot a = \um (h \rhd a). 
\]
We call it the \emph{induced partial action} on $A$ and, in analogy to partial group actions, we will say that this partial action is \emph{admissible} if $B$ is equal to the $H$-submodule generated by $A$. As we have shown in \cite{paper}, all partial Hopf actions are essentially of this kind. In the particular case when the Hopf algebra is the group algebra $\K G$, for some group $G$, the partial Hopf action corresponds to the case of partial group actions where all the ideals $D_g$ are unital, and the action of $\K G$ on $A$ is defined as $g\cdot a=\alpha_g (a1_{g^{-1}})$.

Another important feature of partial group actions which extends to the Hopf algebraic context is the construction of partial crossed products, or partial skew group rings. Basically, given a partial action of a group $G$ on an algebra $A$, the partial skew group ring is the ring generated by finite sums as 
\[
A\rtimes G =\{ \sum_{g\in G} a_g \delta_g |\, a_g \in D_g \} ,
\]
and with the product defined by
\[
(a_g \delta_g )(b_h \delta_h )=\alpha_g (\alpha_{g^{-1}} (a_g )b_h )\delta_{gh} .
\]
Now, given a partial action of a Hopf algebra $H$ on $A$, one can make $A \otimes H$ to  become an associative, possibly non-unital algebra with the product 
\[
(a \otimes h) (b \otimes k) = \sum a(h_{(1)} \cdot b) \otimes h_{(2)} k 
\]
and the left ideal generated by the idempotent $e = \um \otimes 1_H$ is a unital algebra (with unity $e$), called the {\bf partial smash product} of $A$ by $H$, and denoted by $\underline{A \# H}$. As a vector space, it is generated  by elements of the form $\sum a(h_{(1)} \cdot \um) \# h_{(2)}$. In the particular case, when the Hopf algebra is the group algebra $\K G$, the partial smash product is precisely the partial skew group ring under the associated partial group action of $G$ on $A$.

Now, we are going to define the ingredients needed to state the globalization theorem for partial Hopf actions.

\begin{defi} An enveloping action, or globalization, for a partial $H$-module algebra $A$ is a pair $(B,\varphi)$, where $B$ is a  (not necessarily unital) $H$-module algebra and $\varphi: A \vai B$ is an algebra monomorphism such that:
\begin{enumerate}[\bf 1)]
\item $A^\prime = \varphi (A)$ is a right ideal of $B$. 
%\item $B$ is generated by $\varphi(A)$ as an $H$-module. 
\item The map $\varphi$ is an isomorphism of partial $H$-module algebras between $A$ and $A'$ with the partial action induced from $B$, that is:
\begin{equation}\label{morfismo}
\varphi(h \cdot a) = \varphi(\um) (h \rhd \varphi(a))
\end{equation}
\end{enumerate}
\end{defi}
 
We remark that this definition of enveloping action is less restrictive than the one used in \cite{paper}. We also call  $B$ an enveloping $H$-module algebra of $A$.
 
\begin{thm}\cite{paper}\label{enveloping} Every partial $H$-module algebra $A$ has an enveloping action $(B,\varphi)$ such that the induced partial action on $\varphi(A)$ is admissible. 
\end{thm} 

The idea of the proof is to embedd $A$ into $\Hom_\K(H,A)$, which is an $H$-module algebra with the convolution product and the action given by $h \rhd f (k) = f(kh) $. Consider the linear map
 \begin{eqnarray}
 \varphi: A & \vai &  \Hom_\K(H,A) \nonumber \\
 a & \mapsto & \varphi(a): k \mapsto  k \cdot a \nonumber
 \end{eqnarray}
 
 It turns out that $\varphi$ is an algebra monomorphism that also satisfies equation (\ref{morfismo}), and that the submodule algebra $B$ generated by $\varphi(A)$ includes it as an ideal. 

It can be shown that $(B,\varphi)$ is not unique but it is {\it minimal}, in the sense that if $(B^\prime, \varphi^\prime)$ is another enveloping action, then there is a $H$-module algebra epimorphism from $B^\prime$ onto $B$; in the definition of this epimorphism we use the fact that $B = H \rhd \varphi(A)$, i.e., that the submodule generated by $\varphi(A)$ is already a subalgebra of $\Hom_\K(H,A)$. This follows from the equation
\begin{equation}\label{modulogerado}
(h \hits x)(k \hits y)  =   \sum h_{(1)} \hits (x ((S(h_{(2)})k) \hits y))
\end{equation}
which holds in every $H$-module algebra (see \cite{romenos}, Lemma 6.1.3).  

Several results from the classical theory of Hopf actions on algebras have been extended to the partial case \cite{galois,lomp}. In \cite{lomp}, a version of the duality theorem of Cohen and Montgomery for the partial smash product is presented. In this paper, we use Theorem \ref{enveloping} to push this result forward, proving a version of the Blattner-Montgomery theorem for finite-dimensional Hopf algebras.

Hopf algebras are well behaved mathematical objects, because of their good properties with relation to duality, and one can also consider {\bf  partial coactions} of Hopf algebras and partial $H$-comodule algebras as well. 

\begin{defi} \cite{caen06} A (right) {\bf partial $H$-comodule algebra} is defined as an algebra $A$ with a linear map $\prho: A \vai A \otimes H $ such that
\begin{enumerate}[\bf 1)]
\item  $\prho (ab) =\prho (a) \prho (b)$, $\forall a,b\in A$. 
\item  $(I\otimes \epsilon )\prho (a)=a$, $\forall a\in A$. 
\item  $(\prho \otimes I)\prho (a) =(\prho (\um) \otimes 1_H )
((I\otimes \Delta)\prho (a))$, $\forall a\in A$. 
\end{enumerate}
\end{defi}

In what follows, every (partial) $H$-comodule algebra is a right (partial) $H$-comodule algebra. 
In terms of the Sweedler notation, if we denote $\prho(a)$ by  
\[
\prho (a) =\sum a^{[0]} \otimes a^{[1]}
\]
then we may rewrite the conditions above as 

\begin{enumerate}
\item[1)] $\sum (ab)^{[0]} \otimes (ab)^{[1]} =
a^{[0]}b^{[0]} \otimes
a^{[1]}b^{[1]}$, $\forall a,b\in A$,
\item[2)] $\sum a^{[0]} \epsilon(a^{[1]}) =a$, $\forall a\in A$,
\item[3)] $\sum a^{[0][0]}\otimes 
a^{[0][1]} \otimes a^{[1]} = \sum \um ^{[0]}a^{[0]} \otimes 
\um ^{[1]}{a^{[1]}}_{(1)} \otimes {a^{[1]}}_{(2)}$, $\forall a\in A$.  
\end{enumerate}

\begin{defi} We say that an algebra map $\theta: A \vai A^\prime$ between two partial $H$-comodule algebras is a morphism (of  partial $H$-comodule algebras) if the following diagram is commutative:
\[
\xymatrix{
A \ar[d]_{\prho_A} \ar[r]^\theta & A^\prime \ar[d]^{\prho_{A^\prime}}\\
A \otimes H \ar[r]_{\theta \otimes I } & A^\prime \otimes H  
} 
\]
\end{defi}

One may obtain induced partial coactions in the following manner: if $(B,\rho)$ is an $H$-comodule algebra and $A$ is a right ideal with unity $\um$, then the map 
\begin{eqnarray}
\prho: A & \vai &  A \otimes H \nonumber \\
a & \mapsto &  (\um \otimes 1_H) \rho(a) \nonumber
\end{eqnarray}
induces a partial comodule structure on $A$. The same question arises: are all partial $H$-coactions of this kind? An affirmative answer was given in \cite{paper} for finite dimensional Hopf algebras, and here this result is extended to all Hopf algebras. First we explain briefly the elements to state correctly a globalization theorem for partial $H$ coactions.

\begin{defi} \label{envelopingcoaction} An {\bf enveloping coaction}, or globalization, for a partial $H$-comodule algebra $(A,\prho)$ is an $H$-module algebra $(B,\rho)$ and an algebra monomorphism  $\theta: A \vai B$ such that 
\begin{enumerate}[\bf 1)]
\item $\theta(A)$ is a unital right ideal of $B$, 
\item $B$ is generated by $\theta (A)$ as an $H$-comodule, 
\item The diagram below commutes  
\[
\xymatrix{
A \ar[d]_\prho \ar[r]^\theta & B \ar[d]^{(\theta(\um) \otimes 1_H) \rho} \\
A \otimes H \ar[r]_{\theta \otimes I } & B \otimes H  }
\]
i.e., $(\theta \otimes I )\prho(a) = ((\theta(\um) \otimes 1_H) \rho)  \theta(a)$ for each $a \in A$. In other words, $\theta: A \vai \theta(A)$ is a morphism of partial comodule algebras, where $\theta(A)$ has the induced coaction.  
\end{enumerate}
\end{defi}

As it is expected, if $H$ is finite-dimensional it is possible to go from partial $H$-comodules to partial $H$-modules and back, and this was proven by C. Lomp in a previous version of \cite{lomp}. This correspondence can be extended to a more general context. 

We recall that a {\bf pairing} between two Hopf algebras $H_1$ and $H_2$ is a a linear map
\[
\begin{array}{rccc}
\langle , \rangle : & H_1 \otimes H_2 & \rightarrow & \K\\
\, & h\otimes f & \mapsto & \langle h,f \rangle 
\end{array}
\]
such that
\begin{enumerate}
\item[(i)] $\langle hk,f \rangle=\langle h\otimes k,\Delta (f) \rangle$.
\item[(ii)] $\langle h,fg \rangle=\langle \Delta (h),f\otimes g \rangle$.
\item[(iii)] $\langle h,1_{H_2} \rangle =\epsilon (h)$.
\item[(iv)] $\langle 1_{H_1},f \rangle =\epsilon (f)$.
\end{enumerate}
A pairing is said to be {\it nondegenerate} if the following conditions hold:
\begin{enumerate}
\item If $\langle h,f \rangle=0$, for all $f\in H_2$ then $h=0$.
\item If $\langle h,f \rangle=0$, for all $h\in H_1$ then $f=0$. 
\end{enumerate}

Collecting results from \cite{paper} we have 
\begin{thm} \cite{paper}
Let  $\langle , \rangle :  H_1 \otimes H_2  \rightarrow  \K $ be a pairing between the Hopf algebras $H_1$ and $H_2$, and let $A$ be an algebra. 
\begin{enumerate}
\item If $A$ is a partial $H_1$-comodule algebra, then $A$ becomes a partial $H_2$-module algebra via the partial action
\[
   h \cdot a = \sum a^{[0]} \langle a^{[1]} ,h \rangle . 
\]
\item Conversely, if the pairing is \emph{nondegenerate}, suppose that $H_1$ acts
partially on an algebra $A$ in such a manner that $\mbox{dim}(H_1 \cdot a)<\infty$ for all $a \in A$. Then there is a partial $H_2$-coaction  $\prho :A\rightarrow A\otimes
H_2$ defined by
\[
(I\otimes h)\prho (a)= h\cdot a, \qquad \forall h\in H_1
\]
where $I\otimes h :A\otimes H_2 \rightarrow A$ is given by 
\[
(I\otimes h)(\sum_{i=1}^{n} a_i \otimes f_i)=\sum_{i=1}^{n} a_i \langle h,f_i
\rangle.
\]
\end{enumerate}
\end{thm} 			

If $H^\circ$ is the finite dual of the Hopf algebra $H$, then there is the canonical pairing $\langle , \rangle :  H^\circ \otimes H  \rightarrow  \K $ given by $\langle h^\ast, k\rangle = h^\ast(k)$. Therefore a partial $H$-comodule algebra $A$ has the corresponding partial $H^\circ$-module algebra structure given in item (1) above, and  using  Theorem \ref{enveloping} we obtain a global $H^{\circ}$-action on an algebra $B$. In order to go back to $H$-coactions, we need a little bit more. It is known that the canonical pairing of $H^\circ$  and $H$ is nondegenerate if and only if $H^\circ$ separates points; the problem is to assure that the $H^\circ$-module algebra $B$ is rational as $H^\circ$-module. If this is the case, which works at least when either $H$ or $B$ are finite dimensional, then $B$ is also a globalization of the partial $H$ coaction on $A$, as stated in the following theorem.

\begin{thm}\label{globalcoaction}\cite{paper} Let $H$ be a Hopf algebra such that its finite dual $H^\circ$
  separates points. Suppose that $H$ coacts partially on a unital algebra $A$
  with coaction $\prho$ and an 
enveloping action, $(B, \theta)$, of the partial action of $H^\circ$ on $A$ is a
rational left $H^\circ$-module. Then the pair $(B,\theta)$ is an enveloping coaction of $A$. 
\end{thm}

In what follows, we prove that there is an enveloping coaction for {\it any} partial $H$-comodule algebra $A$, for any Hopf algebra $H$. We also show that this new construction coincides with the construction given in Theorem \ref{globalcoaction}  when the hypotheses are satisfied. 

\section{Globalization of partial coactions}

Our aim in this section is, given a partial coaction of a Hopf algebra $H$ on a unital algebra $A$, to construct an $H$-comodule algebra $B$ which contains $A$ as a right ideal and such that the partial coaction coincides with the induced coaction obtained from $B$.

Recall that if $M$ is a right $H$-comodule,  then $M$ is a left $H^\ast$-module  via the action 
$h^\ast \rightharpoonup m = (I \otimes h^\ast) \rho(m) = \sum m^{(0)} h^\ast(m^{(1)}) $, and its $H$-subcomodules coincide with its \emph{rational} $H^\ast$-submodules. 
%When $\B$ is an $H$-comodule algebra and $\C \subset \B$, the $H$-subcomodule algebra generated by $\C$ is the smallest su 

\begin{lemma}\label{lema} Let $\B$ be an $H$-comodule algebra. If $\C \subset \B$ is a subalgebra, the subcomodule algebra $\gen$ generated by $\C$ is the subalgebra generated by $H^\ast \rightharpoonup A$. In other words, the set 
\[
S = \{ h^\ast \rightharpoonup a; a \in A, h^\ast \in H^\ast\}
\]
generates $\gen$ as an algebra. 
\end{lemma}
\begin{proof}
We claim that $H^\ast \rightharpoonup \C$ is a rational $H^\ast$-submodule of $\B$, and hence, that it is an $H$-subcomodule containing $\C$. In fact, it is enough to check that for each $a \in \C$, the cyclic submodule $H^\ast \rightharpoonup a$ generated by $a$ is rational, because
\[H^\ast \rightharpoonup \C = \sum_{a \in \C} H^\ast \rightharpoonup a.\]
Given $a \in \C$, write $\rho(a)$ in the form $\rho(a) = \sum_{i=1}^n a_i \otimes h_i$, where $\{h_1, \ldots, h_n\}$ is a linearly independent subset of $H$. The proof of the fundamental theorem of comodules \cite{romenos} assures that the vector subspace $V_a = \mbox{span}\{a_1, \ldots, a_n \}$ is an $H$-subcomodule, and hence an $H^\ast$-submodule, which contains $a$. Therefore $V_a$ contains $H^\ast \rightharpoonup a$ and this last module is rational. 

Going further, if $W$ is an $H^\ast$-submodule containing $a$, since the elements $h_1, \ldots, h_n \in H$ are linearly independent, one can take functionals $h^\ast_1 , \ldots h^\ast_n \in H^\ast$ such that $h^\ast_i (h_j)=\delta_{ij}$. Hence, each $a_i$ can be written as 
$h^\ast_i \rightharpoonup a $ and this implies that $V_a \subset W$, and therefore $H^\ast \rightharpoonup a  = V_a$. 

Consider now the subalgebra  $\gen$  generated by $H^\ast \rightharpoonup \C$. This algebra may also be generated, as a vector space, by monomials of the form $b_1b_2 \cdots b_n$, where each $b_i$ lies in some $V_a = H^\ast \rightharpoonup a$ and $n \in \N$. 
We claim that $\gen$ is also an $H$-subcomodule of $\B$. 

In fact, if $b \in V_a$ then $\rho(b) \in V_a \otimes H \subset \gen \otimes H$ (since each $V_a$ is a subcomodule of $H^\ast \rightharpoonup A$). Assuming that $\rho(v) \in \gen \otimes H$ for every monomial $v$  of length up to $n-1$, consider now the monomial $b_1b_2 \cdots b_n$. By induction, 
\[\rho(b_2 \cdots b_n) = \sum_{j=1}^m c_j \otimes g_j \in \gen \otimes H \]
 and, since $b_1 \in V_a$ for some $a \in \C$, we conclude that  $\rho(b_1b_2 \cdots b_n) = \rho(b_1) \rho (b_2 \cdots b_n)$ lies in $\gen \otimes H $. 

On the other hand, if $M$ is a subcomodule algebra containing $\C$, then each $V_a$ lies in $M$: in fact, $M$ is an $H^\ast$-module and, if $a \in \C$ then  $h^\ast \cdot a = \sum a^{(0)} h^* (a^{(1)}) \in M $ for every  $h^* \in H^*$. And since $M$ is a subalgebra and $H^\ast \rightharpoonup \C \subset M$, then the subalgebra $\gen$ generated by $H^\ast \rightharpoonup \C$ is contained in $M$. Therefore, $\gen$ is the subcomodule algebra generated by $\C$.
\end{proof}

We remark that when $H^\circ$ separates points, by the Jacobson density theorem, for each $a \in \C$ and $h^\ast \in H^\ast$ there is another functional $k^\ast \in H^\circ$ such that $ h^\ast \rightharpoonup a  = \sum a^{(0)} h^* (a^{(1)})  = k^\ast \rightharpoonup a$. Hence, we can choose to work with $H^\circ \rightharpoonup \C$, instead of $H^\ast \rightharpoonup \C$. In this case, the equation (\ref{modulogerado}) from the former section shows that $H^\circ \rightharpoonup \C$ is already a subcomodule algebra, and hence $\gen = H^\circ \rightharpoonup \C$. 

We can provide now a simple proof of the fact that every partial coaction is induced, thus extending above mentioned Theorem \ref{globalcoaction}. 

 \begin{thm}\label{newglobalization}
 Every partial $H$-comodule algebra $(A, \prho)$ has a globalization $(B,\rho)$. 
 \end{thm}
\begin{proof}
Let $A$ be a partial $H$-comodule algebra via the map $\prho: A \vai A \otimes H$. Consider on $A \otimes H$ the trivial comodule structure $\delta = I \otimes \Delta$, and let $B$ be the subcomodule algebra of $A \otimes H$ generated by $\prho(A)$. Note that $\prho(\um)$ is the unity of $\prho(A)$ and $\prho (A)\subseteq B$. We claim that the pair $(B, \prho )$ is an enveloping coaction for the partial $H$-comodule algebra $A$. First, item (2) of definition \ref{envelopingcoaction} is, by construction,  automatically satisfied by $B$. Rewriting the partial comodule law 
\[
(\prho(\um) \otimes I)(I \otimes \Delta) \prho =  (\prho \otimes I) \prho
\] as 
 \[(\prho(\um) \otimes I) \delta \prho =  (\prho \otimes I) \prho ,\] 
 we see that the following diagram comutes 
\[
\xymatrix{
A \ar[d]_\prho \ar[r]^\prho & B \ar[d]^{(\prho(\um) \otimes 1_H) \delta} \\
A \otimes H \ar[r]_{\prho \otimes I } & B \otimes H  }
\]
This means that $\prho$ works simultaneously as the partial comodule structure of $A$ and as the monomorphism $\theta: A \rightarrow B$ of definition \ref{envelopingcoaction}, and therefore item (3) of the same definition is satisfied (note also that the definition of partial $H$-comodule algebra implies that $\prho$ is, indeed, an algebra monomorphism).  
  
Now it only remains to verify that  item (1) of the same definition of enveloping coaction works for the pair $(B, \prho)$. Proving that 
$\prho(\um ) B \subseteq \prho (A)$ will be sufficient to verify that $\prho (A)$ is a right ideal of $B$ with unit $\prho (\um )$. In fact, if $a \in A$ and $b\in B$ then $\prho (a)b=\prho (a) \prho (\um )b\in \prho (A)$. 

Let us verify first for the generators of $B$. As we have seen earlier in Lemma 
\ref{lema}, the set 
\[
S = \left\{\sum a^{[0]} \otimes a^{[1]}_{(1)} h^* (a^{[1]}_{(2)}); a \in A, h^\ast \in H^\ast\right\}
\] 
generates $B$ as an algebra; multiplying such an element by $\prho(\um ) $ on the left we get
\begin{eqnarray}
\label{inclusao}
\prho (\um )\sum a^{[0]} \otimes a^{[1]}_{(1)} h^* (a^{[1]}_{(2)}) &=& \sum \um^{[0]}a^{[0]} \otimes \um^{[1]}a^{[1]}_{(1)} h^* (a^{[1]}_{(2)})  = \nonumber\\ 
&=& \sum a^{[0][0]} \otimes a^{[0][1]} h^* (a^{[1]}) =\nonumber\\ 
&=& \prho(\sum a^{[0]}h^* (a^{[1]})) .
\end{eqnarray}
Recalling that, if $a\in A$ and $h^\ast \in H^\ast$ then 
\[
\sum a^{[0]} h^\ast(a^{[1]}) = (I \otimes h^\ast) \prho(a)
\]
lies in $A$ (since $A$ is a partial $H$-comodule algebra), one concludes that $\prho(\um)b \in \prho(A)$ for every $b \in S$. Suppose now that $b_1,b_2, \ldots, b_n \in S$ and suppose, by induction, that $\prho (\um )b_1\ldots b_{n-1} =\prho (a) \in \prho (A)$. Then
\[
\prho(\um) b_1 \cdots b_n =\prho (a)b_n=\prho (a)
\prho (\um )b_n=\prho (a) \prho (a')=\prho (aa').
\]
It follows that $\prho(\um) b_1 \cdots b_n \in \prho(A)$, for any product of generators of $B$, and hence that $\prho (\um )B \subseteq \prho (A)$. Therefore, the pair $(B,\prho)$ is an enveloping coaction for the partial $H$-comodule algebra $A$.
\end{proof}

When $H^\circ$ separates points, this construction is essentially the same as the previous construction in \cite{paper}. In fact, it is a routine verification that the map
\begin{eqnarray}
\Psi : A \otimes H & \rightarrow & \Hom_\K (H^\circ, A) \nonumber \\
a \otimes h & \mapsto & a \otimes h : k^\ast \mapsto ak^\ast(h) \nonumber
\end{eqnarray}
is a monomorphism of $H^\circ$-module algebras, where the action of $H^{\circ}$ on $A\otimes H$ is given by $h^\ast \rhd (a\otimes h)=\sum a\otimes h_{(1)} h^\ast 
(h_{(2)})$. The composition of $\Psi$ with $\prho$ provides the map 
$\varphi = \Psi \circ \prho : A   \rightarrow  \Hom_\K (H^\circ, A) $
where $\varphi (a) (k^\ast) = \sum a^{[0]}k^\ast(a^{[1]}) = k^\ast \cdot a$. This is exactly the map $\varphi: A \vai \Hom_\K(H^\circ, A)$ which implements the minimal enveloping action of the partial $H^\circ$-module algebra $A$. It remains to prove that the $H^\circ $-module algebra $B'=H^\circ\rhd \varphi (A)$ is isomorphic to the $H$-comodule algebra $B$ obtained in Theorem \ref{newglobalization}. Since $\Psi$ is a monomorphism of $H^\circ$-module algebras, it is sufficient to verify on the generators:
\begin{eqnarray}
\Psi (\sum a^{[0]} \otimes {a^{[1]}}_{(1)} h^\ast ( {a^{[1]}}_{(2)})) &=&
\Psi (h^\ast \rhd (\sum a^{[0]} \otimes a^{[1]} ))= \nonumber\\
&=& h^\ast \rhd \Psi (\sum a^{[0]} \otimes a^{[1]} )=\nonumber\\
&=& h^\ast \rhd \Psi(\prho (a))=h^\ast \rhd \varphi (a) .\nonumber
\end{eqnarray}
From these equalities, it follows that $\Psi (B) =H^\circ \rhd \varphi (A)=B'$.

\bigskip

We present now some examples of partial coactions and their globalizations. 

\begin{ex}
 This is a variation on a  example from \cite{galois}. Let $G$ be a finite group. If $N$ is a normal subgroup of $G$ with $\mbox{char}(k)\nmid |N|$, then $e_N = \dfrac{1}{|N|} \sum_{n \in N} n$  is a central idempotent in $\K G$. Let $A = e_N \K G$ be the ideal generated by $e_N$. Consider the partial $\K G$-coaction induced on $A$ by $\Delta: \K G \vai \K G \otimes \K G$, i.e., 
\[
\prho(e_N g) = (e_N \otimes 1) \Delta(e_N g) = e_N g \otimes e_N g = \dfrac{1}{|N|^2} \sum_{m,n \in N} mg \otimes ng .
\]
From this coaction, one can see that the associated partial action of $\K G^\ast$ on $A$ is given by
\[
p_h \cdot (e_N g)=[h\in Ng] \dfrac{1}{|N|} e_N g ,
\]
where the expression $[h\in Ng]$ is the Boolean value of this sentence, that is, it is equal to $0$ if $h\notin Ng$ and equal to $1$ if $h\in Ng$.
In order to construct the globalization comodule algebra, we need to characterize its set of generators. Since 
\[
p_h \rightharpoonup \prho(e_N g) = p_h \rightharpoonup (\dfrac{1}{|N|^2} \sum_{m,n \in N} mg \otimes ng)   =  
 \dfrac{1}{|N|^2} \sum_{m,n \in N} mg \otimes ng p_h(ng)  
\]
then, we get 
\[
p_h \rightharpoonup \prho(e_N g) = [h\in Ng ] \dfrac{1}{|N|} e_N g \otimes h.  
\]

Hence, the subcomodule $C_{e_N g}$ generated by $\prho(e_N g)$ is the subspace $C_{e_N g} = \langle e_N g \otimes h | h \in Ng \rangle$. The sum of the subcomodules $C_{e_N g}$ is the subcomodule algebra $B = \langle e_N g \otimes g | g \in G \rangle$, which is isomorphic to $\K G$, as a $\K G$-comodule algebra, via the algebra monomorphism 
 \begin{eqnarray}
 \Phi : \K G & \vai & A \otimes \K G  \nonumber \\
   v & \mapsto &(e_N \otimes 1) \Delta (v) \nonumber 
 \end{eqnarray}
\end{ex}

\begin{ex}
This example of partial coaction comes from \cite{caen06}. Let $H_4$ be the Sweedler algebra $H_4 = \langle 1, c, x, cx \mid c^2=1, x^2 = 0, xc = - cx \rangle $, with Hopf algebra structure given by $\Delta c=c\otimes c$, $\Delta x =x\otimes 1+c\otimes x$, $\epsilon (c)=1$, $\epsilon (x)=0$, $S(c)=c$ and $S(x)=-x$. For any $\alpha \in \K$, the element $f = \frac{1}{2}(1 + c + \alpha cx)$ is an idempotent and, identifying $\K \otimes H_4$ with $H_4$ in the canonical way, the map 
\begin{eqnarray}
\prho: \K & \vai & H_4 \nonumber \\
\lambda  & \mapsto & \lambda f \nonumber 
\end{eqnarray}
defines a structure of partial $H_4$-comodule algebra on $\K$. Now,  
\[
\Delta\prho(\K) = \K (1 \otimes 1 + c \otimes c  + \alpha cx \otimes c + \alpha 1 \otimes cx)
\]
and applying functionals to elements of $\Delta\prho(\K)$, we get the subcomodule 
\[ B = (H_4)^\ast \rightharpoonup \K =  \langle 1, (c + \alpha cx)/2 \rangle = \langle 1, f\rangle\] 
which is also a subalgebra of $H_4$. This is the globalization of the partial coaction $\prho$. 
\end{ex}

\begin{ex}
Consider once more the Sweedler Hopf algebra  $H_4$, and let $A$ be the subalgebra $\K[x]$ of $H_4$.  In \cite{caen06}, it is shown that $A$ is a partial $H_4$-comodule algebra with the coaction
\begin{eqnarray}
\prho(1) & = & \dfrac{1}{2} \left( 1 \otimes 1 + 1 \otimes c + 1 \otimes cx \right) \nonumber \\
\prho(x) & = & (x \otimes 1)\prho(1) =  \dfrac{1}{2} \left( x \otimes 1 + x \otimes c + x \otimes cx \right) \nonumber 
\end{eqnarray}
The subcomodule algebra $B$ of $A \otimes H_4$ generated by $\prho(A)$ has basis 
\[
\beta = \{ 1 \otimes 1,  1 \otimes c + 1 \otimes cx, x \otimes 1,  x \otimes c + x \otimes cx\}
\] 
Denoting $1_B = 1 \otimes 1$, $g = 1 \otimes c + 1 \otimes cx$, and $y =x \otimes 1$, we have that $e= \dfrac{1}{2}(1+g)$ is a central idempotent of $B$ and that $\prho(A)$ is the ideal $eB =\langle e, ey \rangle$; using the same notations, 
$B$ can be described as 
\[
B = \langle 1_B, g, y, gy \mid g^2 = 1_B, y^2 =0, gy=yg \rangle
\]
and from this description one can check that $B$ is isomorphic (as an algebra) to the tensor algebra $\K \Z_2 \otimes  \K[Y]/(Y^2)$ (interestingly enough, although we started in $H_4$, the globalization does not lead back to  it). The global coaction given by $\rho = I \otimes \Delta$ on $B$ is  
\[
\rho(1_B) = 1_B \otimes 1, \ \ \rho (g) = g \otimes c + 1_B \otimes cx, \ \ \rho(y) = y \otimes 1.
\]  
and the partial coaction that it induces on $\prho(A)$ is 
\begin{eqnarray}
(e \otimes 1)\rho(e) & = &  1/2(e \otimes 1 + e \otimes c + e \otimes cx) \nonumber \\
(e \otimes 1)\rho(ey) & = & 1/2(ey \otimes 1 + ey \otimes c + ey \otimes cx) \nonumber \end{eqnarray}

Clearly, $\prho: A \vai B$ is an algebra monomorphism that intertwines the partial coaction on $A$ and the induced partial coaction on $B$.
\end{ex}
%In fact, 
%\begin{eqnarray}
%(I \otimes \Delta)\prho(1) & = & \dfrac{1}{2} \left(1 \otimes 1 \otimes 1 + 1 \otimes g \otimes g 
%+ 1 \otimes gx \otimes g+ 1 \otimes 1 \otimes gx \right) \nonumber \\
%(I \otimes \Delta)\prho(x) & = & \dfrac{1}{2} \left( x \otimes 1 \otimes 1 + x \otimes g \otimes g +
% x \otimes gx \otimes g + x \otimes 1 \otimes gx \right) \nonumber 
%\end{eqnarray}

\section{Duality for Partial Actions}

Classical duality theorems have their origins in the context of operator algebras, in works of Takesaki and colaborators for describing the duality between actions and coactions of locally compact groups on Von Neumann algebras \cite{NT}. Later, this duality for actions and coactions of groups on algebras was considered by Cohen and Montgomery \cite{CM}. Basically, given an algebra $A$ with a left action of $\K G$ on it, there is a natural left action of the dual $\K G^\ast$ on the smash product $A\# \K G$. The Cohen-Montgomery duality theorem states, for $G$ finite, that $(A\#\K G)\# \K G^\ast \simeq A\otimes M_n (\K)$, where $n=|G|$. This kind of result is important since coactions of group algebras correspond to group gradings on algebras.

The extension of this duality theorem to the context of Hopf algebras was made in the work of Blattner and Montgomery \cite{BM}. This theorem states that, given a Hopf algebra $H$, such that its finite dual, the Hopf algebra $H^\circ$, separates points, and a left $H$-module algebra $A$, then the double smash product $(A\# H)\# H^\circ $ is isomorphic to $A\otimes \mathcal{L}$, where $\mathcal{L}$ is a dense subalgebra of $\End_\K(H)$. In the case of a finite dimensional Hopf algebra $H$, this results simplifies to $(A\# H)\# H^\ast \simeq A\otimes \End_\K (H)$.

A new version of the Cohen-Montgomery theorem for the case of partial group actions was proposed by Lomp in \cite{lomp}. Basically, the author obtained, for a finite group $G$ with $|G|=n$, an algebra morphism $\Phi$ from the smash product $(A\rtimes G)\#\K G^\ast$ into the matrix algebra $M_n(A)$ and, as a consequence, a decomposition of this smash product as the direct product of algebras
\[
(A\rtimes G)\#\K G^\ast \simeq \ker \Phi \times e M_n(A) e
\]
where $e = \sum_{g \in G} (g^{-1} \cdot \um) E_{g,g} = \Phi(\um \# 1 \# \epsilon)$. In a previous version of \cite{lomp}, attempting to obtain a version of the Blattner-Montgomery theorem for partial actions of finite dimensional Hopf algebras, the author constructed an algebra morphism from $(\underline{A\# H})\# H^\ast $ into $A\otimes \End_\K (H)$. As we shall see in this section, the globalization theorem for partial Hopf actions explained earlier helps us to obtain the partial versions of both classical duality theorems in a more direct way. We also prove a second version of the Blattner-Montgomery theorem: if $B$ is a globalization of the partial $H$-module algebra $A$, then there is a right $B$-module $M$ such that $(\underline{A \# H})\# H^\ast$ is isomorphic to $\End_B(M)$.  When $H = \K G$, we prove also that $(\underline{A \# \K G} \# \K G^*$ is isomorphic to a matrix ring. 

In what follows, we consider a finite-dimensional Hopf algebra $H$ acting partially on a unital algebra $A$. We have to construct a  $H$-module algebra $B$, with unity $1_B$, having $A$ as a unital ideal such that the partial action on $A$ can be viewed as the partial action induced from $B$. Consider a globalization $B$ for the partial action of $H$ on $A$; without loss of generality, we can consider $A$ as an ideal of $B$. If $B$ is a unital algebra, then we take this admissible globalization; if $B$ doesn't have a unity, we take the unitization $\widetilde{B}=\K\times B$ with the $H$-module structure given by $h\rhd (\lambda ,a)=(\epsilon (h)\lambda ,h\rhd a)$. It is easy to see that $\widetilde{B}$ is an $H$-module algebra. Moreover, by the construction of $\widetilde{B}$ it follows that $A$ is an ideal of this algebra. Finally, as $A$ is embedded into $\widetilde{B}$ by the inclusion $a\mapsto (0,a)$,  the induced coaction $h\cdot (0,a)$, defined as $h\cdot (0,a)=(0,\um)(h\rhd (0,a))$, is equal to 
$(0,\um (h\rhd a))=(0,h\cdot a)$. Hence, the inclusion map intertwins the partial actions. Note that $\widetilde{B}$ is a globalization, but not an admissible globalization of $A$ because $\widetilde{B}$ is not generated by $A$ as $H$-module.

\bigskip 

The partial version of the Blattner-Montgomery theorem intends to characterize the smash product 
$(\underline{A\# H})\# H^\ast$ and, in particular, the partial smash product $\underline{A\# H}$, as subalgebras of $A\otimes \End_\K (H)$. To this end, we shall use the classical result with a (unital) globalization $B$, which gives the isomorphism $(B\# H)\# H^\ast \simeq B\otimes \End_\K (H)$, and then find the suitable idempotents projecting  onto the subalgebras under investigation.

First, we shall state briefly the morphisms involved in the classical Blattner-Montgomery theorem \cite{BM}. 

\begin{lemma} \cite{CM,susan} Let $H$ be a finite dimensional Hopf algebra. Then, the linear maps
\begin{enumerate}
\item $\lambda :H\# H^\ast \rightarrow \End_\K (H)$, defined as, $\lambda (h\# f)(k)=h(f\rightharpoonup k)$,
\item $\rho :H^\ast \# H \rightarrow \End_\K (H)$, defined as, $\rho (f\# h)(k)=(k\leftharpoonup f)h$
\end{enumerate}
for every $h,k\in H$ and $f\in H^\ast$, are isomorphisms of algebras.
\end{lemma}

Let $B$ be a left $H$-module algebra. Since $H$ is finite dimensional, $B$ is rational as a left $H$-module, hence $B$ is also a right $H^\ast$-comodule; in fact, it is easy to see that it is a comodule algebra. The comodule structure is given by 
\[
\delta(b) = \sum b^{(0)} \otimes b^{(1)} = \sum_i (h_i \rhd b) \otimes h_i^\ast
\]
where $\{ h_i \in H\, |\, 1\leq i\leq n\,\}$ is a basis of $H$ and  $\{ h^\ast_i \in H^\ast \, |\, 1\leq i\leq n\,\}$ is its correspondent dual basis.

Once the morphisms $\lambda$ and $\rho$ were stablished, one can define the maps
\begin{eqnarray}
\Phi: (B \# H) \# H^\ast & \longrightarrow & B \otimes \End_\K (H) \nonumber \\ 
b \# h \# f & \mapsto & \sum b^{(0)} \otimes \rho (\sd(b^{(1)})\# 1_H)\lambda (h \# f) \nonumber
\end{eqnarray}
and
\begin{eqnarray}
\Psi: B \otimes \End_\K (H) & \longrightarrow & (B \# H) \# H^\ast \nonumber \\ 
b \otimes T & \mapsto & \sum ((b^{(0)}\# 1_H\# \epsilon )(1_B\# \lambda^{-1}(\rho (b^{(1)}\# 1_H) T)) .\nonumber
\end{eqnarray}
Following the steps given in \cite{CM} one can show that $\Phi$ and $\Psi$ are mutually inverse algebra isomorphisms. Basically, this is a corollary of the Blattner-Montgomery theorem for the case of finite dimensional Hopf algebras. Note that the morphisms $\Phi$ and $\Psi$ originally presented in that paper are slightly different from our definition: our morphism $\Phi$ corresponds to their morphism $\Phi$ composed with $I\otimes \lambda$. 

The next step is to restrict the domain and codomain of the morphism $\Phi$ in order to get the correct subspaces corresponding to the partial action on $A$. First, $B \otimes \End_\K(H)$ can be projected onto $A \otimes \End_\K(H)$
by left (or right) multiplication by $\um \otimes I$, since $\um $ is a central idempotent in $B$. The domain of $\Phi$ can also be restricted to $(\underline{A \# H}) \# H^\ast$ and we get an algebra morphism $\widetilde{\Phi} = (\um \otimes I) \Phi : (\underline{A \# H}) \# H^\ast \vai A \otimes \End_\K(H)$. 

When we multiply by $\um \otimes I$, a non-trivial kernel may appear. One can calculate it in the following way: let $E = \Psi (\um \otimes I)$ and $F= \Psi ((1_B - \um )\otimes I)$. Then $E$ and $F$ are central orthogonal idempotents of $(B \# H) \# H^\ast$ such that $E + F = 1_B \# 1_H \# \epsilon $ (which is the unity of $(B \# H) \# H^\ast$). 

The unity of $(\underline{A \# H}) \# H^\ast$ is the element  $e = \um \# 1_H \# \epsilon $. It is easy to see that $e$ is an idempotent of $(B \# H) \# H^\ast $, and that 
 $(\underline{A \# H}) \# H^\ast = e ((B \# H) \# H^\ast ) e$. We also have
 \[
 e = ee = e (\e e +  \f e) = e \e e + e \f e  
 \]
 and, since $E,F$ are central orthogonal idempotents and $e$ is an idempotent, $e \e e $ and $e \f e$ are orthogonal idempotents of $(\underline{A \# H}) \# H^\ast$.  Therefore one has the decomposition 
\[(\underline{A \# H}) \# H^\ast = e \e e(\underline{A \# H}) \# H^\ast \oplus e \f e(\underline{A \# H}) \# H^\ast
\]
 as an algebra. Note also that since $e$ is an idempotent and $\e$ and $\f$ are central, we may write $e \e e = \e  e$ and $e \f e = \f e $.
 
Applying $\widetilde{\Phi}$ to $v \in (\underline{A \# H}) \# H^\ast$, we have 
\[
\widetilde{\Phi}(v) = (\um \otimes I)\Phi(v) = \Phi(\e)\Phi(\e ev +   \f e v) 
= \Phi(\e e v) = \Phi(\e)\Phi(\e e v) = \widetilde{\Phi}(\e e v).
\]
Hence $\widetilde{\Phi}(v) = 0$ iff $\Phi(v) = \Phi(\f ev)$ and, since $\Phi$ is an isomorphism, $v = \f e v$.  This shows that 
\[\ker \widetilde{\Phi} =  \f e(\underline{A \# H}) \# H^\ast . \]
By the same token, $\widetilde{\Phi}$ restricted to $\e e (\underline{A \# H}) \# H^\ast$ is a monomorphism. 

A natural question is about the necessary and sufficient conditions to have a nontrivial kernel, that is, when the subspace $F e (\underline{A \# H}) \# H^\ast \neq 0$. 

The vector space $B \otimes H$ is a free right $B$-module by $(b \otimes h) c := bc \otimes h$. With this structure on $B \otimes H$, it is well known that the linear map $\eta: B \otimes \End_\K(H) \vai \End_B(B \otimes H)$ defined by $\eta(b \otimes T) (c \otimes k) = bc \otimes T(h)$ is an isomorphism of algebras (and of left $B$-modules). In what follows, we will identify the algebras  $B \otimes \End_\K(H)$ and $ \End_B(B \otimes H)$, and will consider $b \otimes T$ as an endomorphism of the $B$-module $B \otimes H$.

Suppose that the kernel of $\widetilde{\Phi}$ is trivial: since $e$ is the unity of the subalgebra $(\underline{A \# H}) \# H^\ast$, this is equivalent to say that $Fe= (Fe)e=0$, and therefore its image under $\Phi$ vanishes identically as a linear transformation on $B\otimes H$. Let $\{ h_i \in H\, |\, 1\leq i\leq n\,\}$ be a basis of $H$ and  $\{ h^\ast_i \in H^\ast \, |\, 1\leq i\leq n\,\}$ its correspondent dual basis.
We can write explicitly the action of $\Phi(e)$ on $B \otimes H$ by 
\begin{eqnarray}
\Phi(e) (b \otimes k)  & = & 
\sum \um^{(0)} \otimes \rho(\sd(\um^{(1)}) \# 1_H) \lambda( 1_H \# \epsilon) (b \otimes k) \nonumber \\
& = & \sum_{i} (h_i \rhd \um)b \otimes \rho(\sd(h_i^\ast) \# 1_H) \lambda( 1_H \# \epsilon) ( k) \nonumber \\
& = & \sum_{i} (h_i \rhd \um)b \otimes \rho(\sd(h_i^\ast) \# 1_H) ( k) \nonumber \\
& = & \sum_{i} (h_i \rhd \um)b \otimes \sum \langle k_{(1)}, \sd(h_i^\ast) \rangle k_{(2)} \nonumber \\
& = & \sum_{i} ( \sum \langle k_{(1)}, \sd(h_i^\ast) \rangle h_i \rhd \um)b \otimes  k_{(2)} \nonumber \\
& = & \sum (S^{-1}(k_{(1)}) \rhd \um)b \otimes  k_{(2)} .\nonumber 
\end{eqnarray}
If $\Phi (Fe)(b\otimes k)=0$ for every $b\in B$ and $k\in H$, then
\[
(1_B -\um )(\sum (S^{-1}(k_{(1)}) \rhd \um)b \otimes  k_{(2)})=
\sum ((S^{-1}(k_{(1)}) \rhd \um ) -(S^{-1}(k_{(1)}) \cdot \um ))b \otimes  k_{(2)}=0.
\]
Taking $b=1_B$, we have
\[
\sum ((S^{-1}(k_{(1)}) \rhd \um ) -(S^{-1}(k_{(1)}) \cdot \um )) \otimes  k_{(2)}=0.
\]
Finally, aplying $(I\otimes \epsilon)$ to this previous equality, we obtain
\[
S^{-1}(k)\rhd \um =S^{-1}(k) \cdot \um , \ \ \mbox{ for every } k \in H,
\]
and since the antipode is bijective, this implies that 
\begin{equation}\label{trivial}
h\rhd \um = h\cdot \um \ \ \mbox{for every }h\in H.
\end{equation}
This is sufficient to assure that the partial and the global action coincide, i.e, that $A$ is an $H$-\emph{module} algebra. In fact, for $h\in H$ and $a\in A$:
\begin{eqnarray}
h\rhd a &=& h\rhd (\um a)=\nonumber\\
&=& \sum (h_{(1)}\rhd \um )(h_{(2)}\rhd a)= \nonumber\\
&=& \sum (h_{(1)}\cdot \um )(h_{(2)}\rhd a)=\nonumber\\
&=& \sum (h_{(1)}\cdot \um ) \um (h_{(2)}\rhd a)=\nonumber\\
&=& \sum (h_{(1)}\cdot \um )(h_{(2)}\cdot a)=\nonumber\\
&=& h\cdot (\um a) = h\cdot a. \nonumber
\end{eqnarray}
One can also prove that $A$ is an $H$-module algebra using equation (\ref{trivial}) in item (3) of the definition of partial action.  Therefore, the kernel will be trivial if, and only if, the partial action of $H$ on $A$ is a total action. Summing up our findings, we have the following version of the Blattner-Montgomery theorem. 
\begin{thm}
Let $H$ be a finite dimensional Hopf algebra, let $A$ be a partial $H$-module algebra and $(B,\varphi)$ a enveloping action of $A$, where $B$ is a unital algebra. Identifying $A$ with $\varphi(A)$, let $\Phi: (B \# H) \# H^\ast \rightarrow  B \otimes \End_\K (H)  $ be the isomorphism of the Blattner-Montgomery theorem, $\Psi = \Phi^{-1}$, and consider the elements $\e = \Psi(\um \otimes I)$, $\f = \Psi((1_B - \um) \otimes I))$. 
\begin{enumerate}[\bf 1)]
\item The algebra homomorphism $\widetilde{\Phi} = (\um \otimes I) \Phi$ is given by 
\begin{eqnarray}
\widetilde{\Phi}: (\underline{A \# H}) \# H^\ast & \longrightarrow & A \otimes \End_\K (H) \nonumber \\ 
\sum_{(k)} a (k_{(1)} \cdot \um) \# k_{(2)} \# f & \mapsto & \sum_i \sum_{(k)} h_i \cdot a (k_{(1)} \cdot \um) \otimes \rho (\sd(h_i^\ast)\# 1_H)\lambda (k_{(2)} \# f) \nonumber
\end{eqnarray}
\item If $e $ is the unity of $\underline{A \# H}$, then  $e\e e$ and $e \f e$ are orthogonal idempotents of $(\underline{A \# H}) \# H^\ast$, and 
\[(\underline{A \# H}) \# H^\ast = e \e e(\underline{A \# H}) \# H^\ast \oplus e \f e(\underline{A \# H}) \# H^\ast
\]
is a decomposition of the algebra $(\underline{A \# H}) \# H^\ast$ as a direct sum of ideals.
\item $\ker \widetilde{\Phi} =e \f e(\underline{A \# H}) \# H^\ast $, and this kernel is trivial if and only if $A$ is an $H$-module subalgebra of $B$; more precisely, $h \cdot a = h \rhd a$ for all $a \in A$ and all $h \in H$. 
\end{enumerate}
\end{thm}

We now particularize the discussion to group algebras. The previous result allows us to reobtain some of the main results of \cite{lomp}, and its proof provides still another characterization of the algebra $(A\rtimes G) \# \K G^\ast$. In what follows we will use $A \# \K G$ instead of $A\rtimes G$; we also suppose that the partial action is such that each idempotent $g \cdot \um$ is central, as it happens when the  partial $\K G$-action is induced by a partial action of the group $G$ \cite{caen06,paper}.

 When $H$ is the group algebra $\K G$ of a finite group $G$ of order $|G|=n$, acting on an algebra $B$, the classical Cohen-Montgomery theorem \cite{CM} says that 
\[
\begin{array}{rcl} 
\widehat{\Phi}: (B \# \K G) \# \K G^* & \longrightarrow & M_n(B) \\
\sum_{g,h} (b_{g,h} \# g) \#p_h & \mapsto & \sum_{g,h} ((gh)^{-1} \acts b_{g,h}) E_{gh,h}
\end{array}
\]
is a $\K$-algebra isomorphism. This map may be obtained as the composition of the map $\Phi$ of the Blattner-Montgomery with the canonical isomorphism of $B \otimes \End_\K(H)$ and $M_n(B)$ given by $\sum_{g,h} b_{g,h} \otimes e_{g,h} \mapsto \sum_{g,h} b_{g,h} E_{g,h}$, where $e_{g,h}$  is the map that takes $g$ to $h$ and kills every other basis element, and $E_{g,h}$ is the associated matrix.
 
In the case of partial group actions, given an element $\sum_{g,h} a_{g,h}(g \cdot \um) \# g \# p_h \in (\underline{A \# \K G}) \# \K G^*$ we have,  
\begin{eqnarray}
\widetilde{\Phi}\left(\sum_{g,h}a_{g,h}(g \cdot \um) \#  g \# p_h\right) & = & 
(\um I )\cdot  \widehat{\Phi} \left( \sum_{g,h}a_{g,h}(g \cdot \um) \# g \# p_h \right) \nonumber \\
& = & \sum_{g,h}\um \left((gh)^{-1} \rhd (a_{g,h}(g \cdot \um))\right) E_{gh,h} \nonumber \\
& = & \sum_{g,h}(gh)^{-1} \cdot (a_{g,h}(g \cdot \um)) E_{gh,h} \nonumber \\
& = & \sum_{g,h}(h^{-1}g^{-1} \cdot (a_{g,h}))(h^{-1}g^{-1}  \cdot (g \cdot \um)) E_{gh,h} \nonumber \\
& = & \sum_{g,h}(h^{-1}g^{-1} \cdot (a_{g,h})) (h^{-1}g^{-1}\cdot \um))(h^{-1}g^{-1}g   \cdot \um) E_{gh,h} \nonumber \\
& = & \sum_{g,h}(h^{-1}g^{-1} \cdot (a_{g,h})) (h^{-1}  \cdot \um) E_{gh,h} \nonumber \\
& = & \sum_{g,h}(h^{-1} \cdot (g^{-1} \cdot (a_{g,h}))) E_{gh,h} \nonumber 
\end{eqnarray}

On the other hand, we can obtain the idempotent $E=\widehat{\Phi}^{-1} (\um I)$, resulting in $E = \sum_k (k \rhd \um) \# 1_H \# p_k$. For the same element $\sum_{g,h} a_{g,h}(g \cdot \um) \# g \# p_h \in (\underline{A \# \K G}) \# \K G^*$ we have in turn,  
\begin{eqnarray}
eE e \sum_{g,h} a_{g,h}(g \cdot \um) \# g \# p_h & = & (\sum_k (k \cdot \um) \# 1_H \# p_k) 
(\sum_{g,h} a_{g,h}(g \cdot \um) \# g \# p_h) \nonumber \\
& = & \sum_{k,g,h,s} ((k \cdot \um) \# 1_H)p_{s^{-1}}\rhd(a_{g,h}(g \cdot \um) \# g) \# p_{sk} p_h \nonumber \\
& = & \sum_{k,g,h,s} ((k \cdot \um) \# 1_H)(a_{g,h}(g \cdot \um) \# g p_{s^{-1}}(g))  \# p_{sk} p_h \nonumber \\
& = & \sum_{k,g,h} ((k \cdot \um) \# 1_H)(a_{g,h}(g \cdot \um) \# g) \# p_{g^{-1}k} p_h \nonumber \\
& = & \sum_{g,h} ((gh \cdot \um) \# 1_H)(a_{g,h}(g \cdot \um) \# g) \#  p_h \nonumber \\
& = & \sum_{g,h} a_{g,h}(g \cdot \um)(gh \cdot \um) \# g \#  p_h \nonumber 
\end{eqnarray}
If we apply directly $\widehat{\Phi}$ on this element, we get
\begin{eqnarray}
\widehat{\Phi}(eE e \sum_{g,h} a_{g,h}(g \cdot \um) \# g \# p_h) &=& \widehat{\Phi}(\sum_{g,h} a_{g,h}(g \cdot \um)(gh \cdot \um) \# g \#  p_h)=\nonumber\\
&=& \sum_{g,h}(h^{-1} \cdot (g^{-1} \cdot (a_{g,h}))) E_{gh,h} .\nonumber
\end{eqnarray}
In this case,  $(\underline{A \# \K G} \# \K G^\ast)$ decomposes as the direct sum of the ideals
\[
e \e e (\underline{A \# \K G} \# \K G^\ast) = 
\{\sum_{g,h} a_{g,h} (gh \cdot \um)(g \cdot \um) \# g \#  p_h \mid a_{g,h} \in A \}
\]
and 
\[
\ker \widetilde{\Phi} = e \f e (\underline{A \# \K G} \# \K G^\ast) =  
\{\sum_{g,h} a_{g,h}(\um - (gh \cdot \um))(g \cdot \um) \# g \#  p_h \mid a_{g,h} \in A \}.
\]
%Note that the factor $(g \cdot \um)$ multiplying $a_{g,h}$ does not appear in \cite{lomp} because there $a_{g,h}$ is already taken in $D_g = A1_g = A(g \cdot \um)$. 

Besides these versions of the duality theorems of Blattner-Montgomery and Cohen-Montgomery, there are also  ``extrinsic'' versions: the algebra $\underline{A \# H} \# H^*$ corresponds to the endormorphism ring of a module over the \emph{enveloping $H$-module algebra} $B$ (and not over $A$), as we see in the following.

\bigskip

Let $B$ be a $\K$-algebra and $M$ a right $B$-module that decomposes as a finite direct sum $M = \oplus_{i=1}^n M_i$. We can identify the $\K$-vector space $\Hom_B(M_i,M_j)$ of morphisms of right $B$-modules $g: M_i \rightarrow M_j$ with the subspace of $\End_B(M)$ of the morphisms that take $M_i$ into $M_j$ and kill every other summand $M_k$, $k \neq i$. This comes from the induced decomposition in $\End_B(M)$:  if $\iota_j:M_j \vai \oplus_{i=1}^n  M_i$ and $p_j : \oplus_{i=1}^n  M_i \vai M_j$ are respectively the canonical injection and the canonical projection,  then 
\begin{eqnarray}\label{decomposicao}
\End_B(M) \simeq \bigoplus_{i,j} p_i \End_B(M) \iota_j \simeq \bigoplus_{i,j} \Hom_B(M_j,M_i)
\end{eqnarray}
as a $\K$-vector spaces. 

It is known that
\[ L(M) = \{(f_{ij}) \in M_n(\End_B(M)); f_{ij} \in \Hom_B(M_i,M_j); 1 \leq i,j \leq n \}\]
is a subalgebra of $M_n(\End_B(M))$, and that the vector space isomorphisms \linebreak $p_i \End_B(M) \iota_j \simeq \Hom_B(M_j,M_i)$ induce an algebra isomorphism $\End_B(M)\simeq L(M)$.

%We are particulary interested in the vector space $B \otimes H$, which is a free right $B$-module by $(b \otimes h) c := bc \otimes h$. With this structure on $B \otimes H$, we have the canonical isomorphism $\Psi: B \otimes \End_\K(H) \vai \End_B(B \otimes H)$ of algebras (and of left $B$-modules), defined by $\Psi(b \otimes T) (c \otimes k) = bc \otimes T(h)$. In what follows, we identify the elements $(b \otimes T)$ and $\Psi(b \otimes T)$.

Consider now a partial $H$-module algebra $A$ and a  unital enveloping $H$-module algebra $B$. The elements $e = 1_A \# 1_H \# \epsilon$ and $f = (1_B - 1_A) \# 1_H \# \epsilon$ form a complete set of orthogonal idempotents for $(B \# H) \# H^*$, and so do $\Phi(e)$ and $\Phi(f)$ in $\End_B(B \otimes H)$. Hence, equation   (\ref{decomposicao}) implies that  $\Phi(e) \End_B(B \otimes H) \Phi(e)$ is a direct summand of $\End_B(B \otimes H)$, and 
\[
\underline{A \#H} \#H^\ast = e (B \# H) \#H^\ast e \simeq \Phi(e) \Phi((B \# H) \#H^\ast ) \Phi(e)\]
\[ \simeq \Phi(e) \End_B(B \otimes H) \Phi(e) \simeq \End_B (\Phi(e) (B \otimes H))
\]
as algebras. We have thus another version of the Blattner-Montgomery theorem for partial Hopf actions.

\begin{thm}\label{BM2}
Let $H$ be a finite-dimensional Hopf algebra over $\K$, $A$ a partial $H$-module algebra, $B$ a unital enveloping $H$-module algebra of $A$, $e = 1_A \# 1_H \# \epsilon$ and $\Phi: (B \# H) \#H^* \vai \End_B(B \otimes H)$ as before. Then 
\[
\underline{A \# H} \# H^* \simeq \End_B(\Phi(e)(B \otimes H))
\]
as $\K$-algebras.
\end{thm}
In this generality, we can't say much about the $B$-module $M = \Phi(e)(B \otimes H)$: from the expression of $\Phi(e)$ calculated before, we know that $M$ is generated, as a $\K$-vector space, by elements of the form 
\begin{eqnarray}
\sum (S^{-1}(k_{(1)}) \rhd \um)b \otimes  k_{(2)}, \ \ \mbox{ where } k \in H, b \in B.
\end{eqnarray}

In the case of partial group actions we can say a lot more. In what follows, $B$ is a unital enveloping $\K G$-module algebra for a partial $\K G$-module algebra $A$, where every idempotent $\id_g = g \rhd \um$ is central. Letting $\widehat{D}_{g}$ stand for the ideal $\D_g = B \id_g$, it follows at once that  
\begin{eqnarray}
M = \bigoplus_g (\widehat{D}_{g^{-1}} \otimes  g) \simeq \bigoplus_g \widehat{D}_{g} 
\end{eqnarray}
as left $B$-modules, and therefore 
\begin{eqnarray}\label{primeiroisoKG}
\underline{A \# \K G} \# \K G^\ast \simeq \End_B(\bigoplus_g \widehat{D}_{g})
\end{eqnarray}
as algebras. This last isomorphism yields a nice description of $\underline{A \# \K G} \# \K G^\ast$ as a matrix algebra, as we see next.

\bigskip

We have used above the algebra isomorphism $\End_B(M) \simeq L(M)$ when $M = \oplus_{i=1}^n M_i$. There is a variation on this theme that will be needed: when each $M_i$ is an  \emph{ideal} of $B$, it is easy to see that 
\[
S(M) = \{(a_{i,j}) \in M_n(B); a_{i,j} \in M_i M_j\}
\]
is a (possibly non-unital) subalgebra of $M_n(B)$; if each $M_k$ is a unital ideal with unity $\id_k$, then $S(M)$ is a unital algebra with unity $\sum_{k=1}^n \id_k E_{k,k}$.

Consider now the module $M= \oplus_g \D_g$. The ideals $ \D_g = B \id_g$  satisfy 
\[\D_g\D_h = \D_h\D_g = \D_g \cap \D_h = B \id_g \id_h.\]
Note also that the ideals $D_g$ of the partial action are given by $D_g = A \D_g = \D_\id \D_g$ (where $\id$ stands for the identity of $G$). Let $e = \um \# \id \# \epsilon = \sum_g \um \#\id \# p_g$ be the unity of $\underline{A \# \K G} \# \K G^\ast$. From the Cohen-Montgomery theorem, it follows that  \[
\underline{A \# \K G} \# \K G^\ast  = e ((B \# \K G) \# \K G^\ast) e \simeq \widehat{\Phi}(e) M_n(B) \widehat{\Phi}(e)\]
and a generic element of the image is 
\begin{eqnarray}
\widehat{\Phi}(e) \sum_{g,h} b_{g,h} E_{g,h} \widehat{\Phi}(e) & = & \sum_{g,h,k,s} \id_{k^{-1}}\id_{s^{-1}} b_{g,h} E_{k,k} E_{g,h} E_{s,s}
\nonumber \\
& = &   \sum_{g,h} \id_{h^{-1}} \id_{g^{-1}} b_{g,h}  E_{g,h} \nonumber 
\end{eqnarray}
Hence, the algebra  $\underline{A \# \K G} \# \K G^\ast $ is isomorphic to the matrix algebra
\[
S(M) = \{ (a_{g,h}) \in M_n(B) | a_{g,h} \in \D_{g^{-1}} \D_{h^{-1}}\}.
\]
 Putting this isomorphism together with the isomorphism $L(M) \simeq \End_B(M)$ and the ones presented in Theorem \ref{BM2} and  equation (\ref{primeiroisoKG}), we have the following result: 

\begin{thm}
Let $G$ be a finite group, $A$ a partial $\K G$-module algebra, $B$ a unital enveloping $\K G$-module algebra of $A$. Identifying $A$ with $\varphi(A)$, let  $\id_g = g \rhd \um$, consider the ideals $\D_g = B \id_g$ as right $B$-modules and let $M = \underset{g \in G}{\bigoplus} \D_{g} $.  Then   
\[
 \underline{A \# \K G} \# \K G^\ast \simeq S(M)\simeq L(M)  \simeq \End_B (M)
\] 
as $\K$-algebras. 
\end{thm}
We remark that the middle isomorphism can be obtained directly: since each $\D_g$ is a cyclic right $B$-module (generated by $\id_g$), there is an  isomorphism of $\K$-vector spaces $\Hom_B(\D_g, \D_h ) \simeq \D_g  \D_h$ given by $f \mapsto f(\id_g)$, which induces an algebra isomorphism from $S(M)$ to $L(M)$.

\end{document}